\documentclass[12pt]{amsart}
\usepackage[english]{babel}
\usepackage{amssymb,amscd}

\usepackage[all]{xy}

\usepackage{euscript}
\usepackage{amsmath}
\usepackage{ifpdf}

\usepackage{amsfonts}
\usepackage{mathrsfs}

\usepackage{amsmath}

\newtheorem{Theorem}{Theorem}[section]
\newtheorem{Lemma}[Theorem]{Lemma}
\newtheorem{Corollary}[Theorem]{Corollary}
\newtheorem{Proposition}[Theorem]{Proposition}

\newtheorem{Remark}[Theorem]{Remark}
\newtheorem{Example}[Theorem]{Example}

\newtheorem{Definition}[Theorem]{Definition}

\newtheorem{Question}[Theorem]{Question}
\newtheorem{Conjecture}[Theorem]{Conjecture}
\newcommand{\M}{\mathfrak{m}}

\def\sqr#1#2{{\vcenter{\hrule height.#2pt
\hbox{\vrule width.#2pt height#1pt \kern#1pt \vrule width.#2pt}
\hrule height.#2pt}}}

\def\qed{\hspace*{\fill} $\square$}

\begin{document}

\title[Dao's question on fullness]{Dao's question on the asymptotic behaviour of fullness}

\author{Cleto B.~Miranda-Neto}
\address{Departamento de Matem\'atica, Universidade Federal da
Para\'iba, 58051-900 Jo\~ao Pessoa, PB, Brazil.}
\email{cleto@mat.ufpb.br}

\author{Douglas S.~Queiroz}
\address{Departamento de Matem\'atica, Universidade Federal da
Para\'iba, 58051-900 Jo\~ao Pessoa, PB, Brazil.}
\email{douglas.queiroz@academico.ufpb.br; douglassqueiroz0@gmail.com}


\begin{abstract} For a local ring $(R, \M)$ of infinite residue field and positive depth, we address the question raised by H. Dao on how to control the asymptotic behaviour of the  $\M$-full, full, and weakly $\M$-full properties of certain ideals (such notions were first investigated by D. Rees and J. Watanabe), by means of bounding appropriate numbers which express such behaviour. We establish upper bounds, and in certain cases even formulas for such invariants. The main tools used in our results are reduction numbers along with Ratliff-Rush closure of ideals, and also the Castelnuovo-Mumford regularity of the Rees algebra of $\M$. 
\end{abstract}

\maketitle

\section{Main preliminaries and motivation: Fullness and Dao's question}

Throughout this note, by a {\it ring} we mean a commutative, Noetherian, unital ring. For a local ring $(R, \M)$, the class of $\M$-full ideals was introduced by D. Rees (unpublished), and its investigation was later formalized and developed by J. Watanabe in the remarkable sequence of papers \cite{Watanabe1, Watanabe2, Watanabe3}. Subsequently, the related classes of full and weakly $\M$-full ideals also appeared (along with other similar classes such as that formed by Burch ideals \cite{Dao2}, completely $\M$-full ideals \cite{Harima-Watanabe}, and basically full ideals \cite{HRR}) and have received attention in the last decade due to their numerous interesting properties; we refer, additionally, to \cite{Celikbas, Dao, GH, Harima-Watanabe2, Rush} and some of their suggested references. There is also an unexpected connection to the classical Zariski-Lipman conjecture (about derivations and smoothness) in the open case of surfaces; see \cite[Section 4]{Cleto}.

Let us invoke the main definitions to be studied here, which are in fact classical notions. Because some authors require the element $x\in \M$ in parts (a) and (b) below to be {\it general}, and although $\M$-fullness and  fullness can still be defined when the residue field is finite (by means of a standard trick; see \cite[Remark 3.3]{Dao}), we shall assume  throughout this paper that the residue class field of any given local ring $(R, \M)$ is infinite, along with the expected condition $x\notin {\M}^2$ in (a) and (b).

\begin{Definition} \rm Let $I$ be an ideal of a local ring $(R, \M)$ with infinite residue field.

\begin{enumerate}
    \item[(a)] $I$ is {\it $\M$-full} if $I\M : x = I$ for some element $x \in \M \setminus {\M}^2$;
    \item[(b)] $I$ is {\it full} if $I : x = I : \M$ for some element $x \in \M \setminus {\M}^2$;
    \item[(c)] $I$ is {\it weakly $\M$-full} if $I\M : \M = I$.
\end{enumerate}
\end{Definition}

\begin{Remark}\label{equiv}\rm Clearly, $\M$-full ideals are weakly $\M$-full. Also, if $I$ is $\M$-primary, then $I$ is weakly $\M$-full if and only if $I$ is basically full in the sense of \cite{HRR}. Moreover, $\M$-full ideals satisfy the so-called {\it Rees property}, and if $(R, \M)$ is a normal domain then any integrally closed ideal is $\M$-full; for details, see \cite{Watanabe1} (also \cite{GH}).
    
\end{Remark}

Next we introduce Dao's invariants (driven by \cite[Theorem 3.1 and Corollary 3.2]{Dao}) which provide a measure of the asymptotic behaviour of the above properties. For a local ring $(R, \M)$, recall $\mathrm{depth}\,R >0$ means that $\M$ contains a regular element. 

\begin{Definition} \rm(\cite[Definition 3.5]{Dao})  Let $(R, \M)$ be a local ring with infinite residue field and $\mathrm{depth}\,R > 0$. Set
\begin{eqnarray}
n_1(I) & = & {\rm min} \{t \geq 0 \mid \mbox{$I\M^{n}$  is $\M$-full  for  all  $n \geq t$\}}; \nonumber \\
n_2(I) & =  & {\rm min} \{t \geq 0 \mid \mbox{$I\M^{n}$  is full  for  all  $n \geq t$\}}; \nonumber \\
n_3(I)  & = & {\rm min} \{t \geq 0 \mid \mbox{$I\M^{n}$  is weakly $\M$-full  for  all  $n \geq t$\}}. \nonumber
\end{eqnarray}
\end{Definition}

Now we recall Dao's question, which motivated us in this note.

\begin{Question}{\rm (}\cite[Question 4.5]{Dao}{\rm )} \label{qdao} \rm Can we find good lower and upper bounds for the $n_{i}(I)$'s? 
\end{Question}

In his paper, Dao also shares his thought that it is not clear how to answer this question even when $R$ is regular and
2-dimensional, in which case he was able to establish the equalities  $n_{1}(I)=n_{2}(I)=n_{3}(I)$; see \cite[Proposition 4.4]{Dao}, which we generalize in Proposition \ref{propdesigualdade}, where we prove that $$n_2(I) \leq n_3(I)= n_1(I)$$ in arbitrary dimension and positive depth (in order to retrieve Dao's proposition, we also apply Lemma \ref{reg-dim2}).

Our main goal in this note is to tackle Dao's problem 
by employing the notions of Ratliff-Rush closure and reduction number of an ideal as well as the Castelnuovo-Mumford regularity of the Rees algebra of $\M$ as key tools. We provide upper bounds for the $n_{i}(I)$'s (or even a formula for $n_{1}(I)$ and $n_{3}(I)$), where in our main results $I$ is assumed to be a reduction of $\M$, while $R$ need not be neither regular nor 2-dimensional. Half of Dao's problem, more precisely the issue of finding good {\it lower} bounds for his numbers, seems to remain untouched -- notice, however, that by virtue of Proposition \ref{propdesigualdade} it suffices to find lower bounds for $n_2(I)$.

In our approach via reduction numbers and the Ratliff-Rush closure operation, our main result is Theorem \ref{cotred}, where we introduce a number $\alpha$ which in fact we prove to coincide with $n_1(I)$ and $n_3(I)$ (hence is an upper bound for $n_2(I)$). A curious consequence described in Corollary \ref{corindep} is the fact that, if $R$ is 1-dimensional and Cohen-Macaulay (and under the assumption that the reduction $I$ is minimal), the number $\alpha$ does not depend on $I$ itself, and in this case agrees more precisely with the (absolute) reduction number of $\M$. In arbitrary dimension, we propose Conjecture \ref{ourconject}.

Furthermore, we approach Question \ref{qdao} via another numerical invariant which is classical in commutative algebra and algebraic geometry. More precisely, our Theorem \ref{cororeg} shows that the Castelnuovo-Mumford regularity of the Rees algebra $\bigoplus_{i\geq 0}{\M}^i$ is an upper bound for the $n_i(I)$'s whenever $I$ is a reduction of $\M$. When $R$ is regular, we prove in Corollary \ref{regular} that all the $n_i(I)$'s are necessarily equal to zero.

In the last section, we furnish three examples which explicitly illustrate our main results. They show, in particular, that the cases $n_2(I)<n_3(I)$ and $n_2(I)=n_3(I)$ are possible, and that in 
 Corollary \ref{corindep} the hypothesis of $I$ being minimal cannot be dropped.

\section{Generalization of a result of H. Dao}\label{main-results}

Our first observation is an easy lemma which, along with Remark \ref{equiv}, extends \cite[Proposition 1.1(d)]{Hong} (where $I$ is assumed to be $\M$-primary). The lemma will be useful to the main result of this section (Proposition \ref{propdesigualdade}).

\begin{Lemma} \label{lemahong}
Let $(R, \M)$ be a local ring with infinite residue field. Let $I$ be an ideal of $R$ and $n \geq 0$.  The following assertions are equivalent:
\begin{itemize}
    \item[(a)] $I\M^{n + 1}$ is full and $I\M^{n}$ is weakly $\M$-full;
    \item[(b)] $I\M^{n}$ is $\M$-full.
\end{itemize}
\end{Lemma}
\begin{proof} First, suppose (a) holds. Since $I\M^{n + 1}$ is full, we can write $I\M^{n + 1} : x = I\M^{n + 1} : \M$
for an element element $x \in \M \setminus {\M}^2$. Now, $I\M^{n}$ being weakly $\M$-full yields $I\M^{n + 1} : \M = I\M^{n}$.
It follows that $I\M^{n + 1} : x = I\M^{n}$, which means $I\M^{n}$ is $\M$-full.

Conversely, assume that $I\M^{n}$ is $\M$-full. Then, $I\M^{n}$ is weakly $\M$-full, and for some element $x \in \M \setminus {\M}^2$ we have $$I\M^{n + 1} : x = I\M^{n} = I\M^{n + 1} : \M,$$
    that is, $I\M^{n + 1}$ is full, as needed. \qed
\end{proof}

\medskip

The main result of this section shows, in particular, that $n_{1}(I)$ and $n_{3}(I)$ coincide for any ideal $I$ in a local ring with infinite residue field and positive depth; as a corollary, we will see that this result generalizes, to such a huge class of rings, a result of Dao established for 2-dimensional regular local rings.

\begin{Proposition}\label{propdesigualdade}
    Let $(R, \M)$ be a local ring with infinite residue field and $\mathrm{depth}\,R > 0$. For any  ideal $I$ of $R$, we have $$n_2(I) \leq n_3(I)= n_1(I).$$ In particular, $\M$-fullness and weak $\M$-fullness share the same asymptotic behaviour.
\end{Proposition}
\begin{proof} For convenience, we divide the proof into two steps.
\begin{itemize}
    \item {\bf Claim 1.} $n_{2}(I)\leq n_{3}(I).$
\smallskip

 To prove this, suppose by way of contradiction that $n_2(I) > n_3(I)$. Set $n = n_1(I)$.  By \cite[Remark 3.6]{Dao}, we can write $n  \geq n_2(I)$. Therefore, $$n - 1 \geq n_3(I)$$ and it follows that $I\M^{n - 1}$ is weakly $\M$-full. Now, as $n\geq n_2(I)$, we get that $I\M^{n}$ is full and hence Lemma \ref{lemahong} ensures that $I\M^{n - 1}$ is $\M$-full, which  contradicts the definition of $n_1(I)$. 

\medskip

\item {\bf Claim 2.} $n_{3}(I) = n_{1}(I).$
         
\smallskip

Keeping in mind the previous step, we can take an integer $$s \geq n_{3}(I) \geq n_2(I).$$ In particular, $I\M^{s + 1}$ is full and $I\M^{s}$ is weakly $\M$-full. According to Lemma \ref{lemahong}, the ideal $I\M^{s}$ must be $\M$-full  for all $s \geq n_{3}(I)$. In other words, $n_{3}(I) \geq n_{1}(I)$. On the other hand, as pointed out in Remark \ref{equiv}, the ideal
$I\M^{s}$ is weakly $\M$-full whenever it is  $\M$-full, i.e.,  $n_{1}(I) \geq n_{3}(I)$, which forces the asserted equality. \qed
    \end{itemize} 
\end{proof}

\begin{Remark} \rm The proposition above gives $$n_{1}(I) = n_3(I) = \mathrm{max}\{n_{2}(I), n_{3}(I)\}$$ and consequently refines \cite[Remark 3.6]{Dao} (but no condition on the residue field is assumed therein).
\end{Remark}

\begin{Lemma}{\rm (}\cite[Proposition 14.1.6]{Swanson-Huneke}, \cite[Proposition 4.1]{Dao}{\rm )} \label{reg-dim2} If $R$ is a $2$-dimensional regular local ring and $I$ is any ideal of $R$, then $$n_1(I) = n_2(I).$$
\end{Lemma}

As a byproduct, in case the residue field is infinite, our Proposition \ref{propdesigualdade} (with the aid of 
Lemma \ref{reg-dim2}) retrieves the following result of Dao.  

\begin{Corollary}{\rm (}\cite[Proposition 4.4]{Dao}{\rm )} \label{Dao-dim2} If $R$ is a $2$-dimensional regular local ring and $I$ is any ideal of $R$, then $n_1(I) = n_2(I)= n_3(I).$
\end{Corollary}

We close this section with a class of examples where the three invariants agree in arbitrary  dimension (note the zero-dimensional case is trivial since an Artinian regular local ring is just a field).

\begin{Example}\label{examplesec1}\rm Let $R$ be a regular local ring with infinite residue field and $Q$ be any parameter ideal of $R$. By \cite[Corollary 4.2]{Hong}, we have  $n_1(Q) = n_2(Q)$. Now, using our Proposition \ref{propdesigualdade} we obtain $$n_1(Q) = n_2(Q) = n_3(Q).$$ 
\end{Example}

\begin{Question}\rm What can be said about the $n_i(I)$'s if the local ring $R$ is still regular, but with ${\rm dim}\,R\neq 2$, and $I$ is any non-parameter ideal of $R$? In particular, can the inequality $n_3(I)\geq n_2(I)$ (obtained in  Proposition \ref{propdesigualdade}) be strict?
\end{Question}

In case $R$ is non-regular  with ${\rm dim}\,R= 2$, we shall illustrate
in Example \ref{illust} that we can have $n_3(I)> n_2(I)$ at least if $R$ is a rational singularity.  

\section{Answers to Dao's question}

In this section we furnish answers to Question \ref{qdao}, concerning again upper bounds, in the case where $I$ is a reduction of the maximal ideal $\M$ of the local ring $R$. The strategy is to relate the $n_i(I)$'s to two (celebrated) numerical invariants:  reduction number, and the Castelnuovo-Mumford regularity; the former is applied along with the theory of Ratliff-Rush closure of ideals.

\subsection{Approach via reduction numbers} We start with the following 
basic definition.

\begin{Definition}\rm Let $J$ be a proper ideal of a ring $R$. An ideal $I \subset J$ is said to be a \textit{reduction of $J$} if $IJ^{r} = J^{r + 1}$ for some integer $r\geq 0$. Such a reduction $I$ is \textit{minimal} if it is minimal with respect to inclusion. If $I$ is a reduction of $J$, we define the \textit{reduction number of $J$ with respect to $I$} as the number
$${\rm r}_{I}(J) =  \mathrm{min}\,\{m \in \mathbb{N} \, \mid \, IJ^{m} = J^{m + 1}\},$$ and the \textit{reduction number of $J$} as
$${\rm r}(J) =  \mathrm{min}\,\{{\rm r}_{I}(J) \, \mid \, \mbox{$I$ is a minimal reduction of $J$}\}.$$ We say that
${\rm r}(J)$ is \textit{independent} if ${\rm r}(J) = {\rm r}_I(J)$ for every minimal reduction $I$ of $J$.
\end{Definition}

Now recall that, for a given ideal $I$ of a ring $R$, the \textit{Ratliff-Rush closure} $\widetilde{I}$
of $I$ is given by
$$\widetilde{I}  =  \bigcup_{n \geq 1}\,I^{n+1} : I^{n}.$$ This is an ideal of $R$ containing $I$ which in fact refines the integral closure of $I$, so that $\widetilde{I}=I$ whenever $I$ is integrally closed. For details on the theory, we refer to 
\cite{Ratliff-Rush}. 

Now suppose $I$ contains a regular element, i.e., a non-zero-divisor on $R$. Then it is well-known that $\widetilde{I}$ is the largest ideal that shares with $I$ the same sufficiently high powers; hence, $$\widetilde{I^{m}} = I^{m} \quad \mbox{for\, all} \quad m \gg 0.$$ This enables us to consider the following helpful number (inspired by \cite[Proposition 4.2]{Rossi-Swanson}).

\begin{Definition}\rm
If $I$ contains a regular element, we set $$s(I)  =  \mathrm{min}\,\{n \in \mathbb{N} \, \mid \, \widetilde{I^{i}} = I^{i} \ \mathrm{for \ all} \ i \geq n \}.$$
\end{Definition}

Note that, since the equality $\widetilde{I^{i}} = I^{i}$ holds trivially for $i=0$, we have that $s(I)\geq 0$ if and only if $s(I)\geq 1$. Thus we can establish that, throughout the entire paper, $s(I) \geq  1$.

\smallskip

The theorem below is our main result and deals with the case where $I$ is a reduction of $\M$. In particular, it highlights the role played by the numbers ${\rm r}_{I}(\M)$ and $s(\M)$. One of the tools is the following basic lemma, which is a special case of \cite[Proposition 2.2]{Nag}.

\begin{Lemma}\label{chain} Let $R$ be a local ring and $J$ a proper ideal of $R$ containing a regular element. Then, for an integer $n \geq 1$, the ideal $\widetilde{J^{n}}$ is the eventual stable value of the increasing sequence
		$$J^{n + 1} : J  \subseteq  J^{n + 2} : J^{2}  \subseteq  J^{n + 3} : J^{3}  \subseteq  \cdots$$
\end{Lemma}

\begin{Theorem}\label{cotred} Let $(R, \M)$ be a local ring with infinite residue field and ${\rm depth}\,R>0$, and let $I$ be a reduction of $\M$. Set $\alpha = \mathrm{max}\{{\rm r}_{I}(\M), s(\M) - 1\}$. Then 
	$$n_2(I) \leq n_3(I)  = n_1(I) = \alpha.$$

\end{Theorem}
\begin{proof} We know from Proposition \ref{propdesigualdade} that $n_2(I) \leq n_3(I)  = n_1(I)$ and so, to conclude, we will show the equality $n_3(I) = \alpha$. First, let us prove $n_3(I) \leq \alpha$. We claim that
\begin{equation}\label{claim} \widetilde{\M^{n + 1}} : \M = \widetilde{\M^{n}} \quad \mbox{for \,all} \quad n\geq 0.
\end{equation} To show this, applying  Lemma \ref{chain} we obtain, for  all $j\gg 0$, $$\widetilde{{\M}^{n}}  = {\M}^{n + j + 1} : {\M}^{j + 1} \quad \mbox{and} \quad \widetilde{{\M}^{n + 1}} = {\M}^{n + j + 1} : {\M}^{j}.$$  Therefore, we can write $$\widetilde{{\M}^{n + 1}} : {\M}  =  ({\M}^{n + j + 1} : {\M}^{j}) : \M  =  {\M}^{n + j + 1} : {\M}^{j + 1}  =  \widetilde{{\M}^{n}},$$ thus establishing (\ref{claim}), which will be freely used in the rest of the proof.

Now, if we take $j \geq \alpha$, then 
$$I\M^{j + 1} : \M =\M^{j + 2} : \M = \widetilde{\M^{j + 2}} : \M = \widetilde{\M^{j + 1}} = \M^{j + 1} = I\M^{j},$$ that is, $I\M^{j}$ is  weakly $\M$-full for all $j \geq \alpha$. This gives $n_{3}(I) \leq \alpha$; next, we show equality must hold.
\begin{itemize}
    \item {\bf Case 1.} First, assume  ${\rm r}_{I}(\M) > s(\M) - 1$. Then, ${\rm r}_{I}(\M) = \alpha$ and we have $$I\M^{\alpha - 1} \neq \M^{\alpha} = \widetilde{\M^{\alpha}} = \widetilde{\M^{\alpha + 1}} : \M  = \M^{\alpha + 1} : \M = I\M^{\alpha} : \M,$$ which shows that $I\M^{\alpha - 1}$ is not weakly $\M$-full, and consequently $n_{3}(I) \geq \alpha$. Since, $n_{3}(I) \leq \alpha$, equality must hold.

\medskip

\item {\bf Case 2.} Finally, if ${\rm r}_{I}(\M) \leq s(\M) - 1$, then $\alpha = s(\M) - 1$. We have shown $n_3(I) \leq \alpha$. Now suppose by way of contradiction that $n_{3}(I) \leq \alpha - 1$. Thus, $$I\M^{\alpha - 1} = I\M^{\alpha} : \M = \M^{\alpha + 1} : \M =  \widetilde{\M^{\alpha + 1}} : \M = \widetilde{\M^{\alpha}},$$ which yields $\widetilde{\M^{\alpha}} = I\M^{\alpha - 1} \subseteq \M^{\alpha }$, that is, $\widetilde{\M^{\alpha}} = \M^{\alpha}$. This is a contradiction, because $s(\M)= \alpha +1$. It follows that $n_3(I) = \alpha$, as asserted. \qed \end{itemize}
\end{proof}

\begin{Remark}\rm Under the hypotheses and notations of Theorem \ref{cotred}, we immediately derive that if $n_2(I) \geq \alpha$ then all the $n_i(I)$'s coincide (with $\alpha$).

\end{Remark}

A particular facet of our theorem is concerned with the interplay between $\M$-fullness (precisely, the number $n_1(I)$) and the operation of Ratliff-Rush closure (by means of $s(\M)$). Here it should be mentioned that other connections are known. Indeed, recall first that the associated graded ring
$$\mathcal{G}(\M)=\bigoplus_{i\geq 0}{\M}^i/{\M}^{i+1}$$
has positive depth if and only if all powers of $\M$ are $\M$-full (see, e.g., \cite[Proposition 3.3(i)]{Watanabe3}). On the other hand, under suitable conditions, the property $\mathrm{depth}\,\mathcal{G}(\M)>0$ is also equivalent to all powers of $\M$ being Ratliff-Rush closed (see \cite[Remark 1.6]{Rossi-Swanson}). In other words,  $n_1(\M) = 0$ if and only if $s(\M) = 1$. Now we point out that Theorem \ref{cotred} reveals this equivalence to be more general, as the following corollary shows.

\begin{Corollary} Let $(R, \M)$ be a local ring with infinite residue field and ${\rm depth}\,R>0$. Then, 
	$$n_2(\M) \leq n_3(\M)  = n_1(\M) = s(\M) - 1.$$
\end{Corollary}
\begin{proof} Clearly, the ideal $\M$ is a reduction of itself with reduction number zero, and hence $\alpha = s(\M) - 1$. Now the case $I=\M$ of Theorem \ref{cotred} gives the result. \qed
\end{proof}

\medskip

Another byproduct of Theorem \ref{cotred} is the following curious fact in dimension 1; it shows that, in this case, the formula for $n_1(I)$ and $n_3(I)$ given in our theorem does not depend on $I$ itself.

\begin{Corollary} \label{corindep}
Let $(R, \M)$ be a $1$-dimensional Cohen-Macaulay local ring with infinite residue field. Then, for any minimal reduction $I$ of $\M$, $$n_2(I)\leq n_3(I) = n_1(I) = {\rm r}(\M).$$ 
\end{Corollary}
\begin{proof} First recall that the analytic spread of $\M$, denoted by $\ell (\M)$, is the Krull dimension of the associated graded ring $\mathcal{G}(\M)$. It is well-known that $I$ is minimally generated by $\ell(\M)$ elements. On the other hand, we have $$\ell(\M)={\rm dim}\,R=1,$$ and hence the minimal reduction $I$ is principal. In this case, ${\rm r}(\M)$ is independent (see
 \cite[page 504]{Huckaba}), so ${\rm r}_I(\M)={\rm r}(\M)$. On the other hand, using again that $I$ is principal, we can apply \cite[Proposition 4.2(i)]{Rossi-Swanson} to obtain $$s(\M) \leq {\rm r}(\M),$$ which yields $\alpha ={\rm r}(\M)$. Now the result follows by Theorem \ref{cotred}. \qed
\end{proof}

\medskip

Now let us emphasize that, in the case of a Cohen-Macaulay local ring $(R, \M)$ with infinite residue field and arbitrary positive dimension, we are not aware (after calculating numerous examples and searching the literature) of any example violating $$s(\M) \leq {\rm r}_I({\M}) + 1,$$ where $I$ is a minimal reduction of $\M$. Thus, supported by Theorem \ref{cotred}, and also inspired by the behaviour in dimension 1 provided by Corollary \ref{corindep}, we can propose the following conjecture.

\begin{Conjecture}\label{ourconject} Let $(R, \M)$ be a Cohen-Macaulay local ring with infinite residue field and ${\rm dim}\,R\geq 2$. Then, for any minimal reduction $I$ of $\M$, $$n_3(I) = {\rm r}_I(\M).$$ 
    
\end{Conjecture}

\subsection{Approach via Castelnuovo-Mumford regularity} For the next corollaries of Theorem \ref{cotred} we need to invoke one more crucial invariant. Let $S = \bigoplus_{n \geq 0}S_{n}$ be a finitely generated standard graded algebra over a ring $S_{0}$ (as usual, by {\it standard} we mean that $S$ is generated by $S_1$ as an $S_0$-algebra) and write $S_{+} = \bigoplus_{n \geq 1}S_{n}$ for the ideal generated by all elements of $S$ of positive degree. For a graded $S$-module $A=\bigoplus_{n \in {\mathbb Z}}A_{n}$ satisfying $A_n=0$ for all $n\gg 0$, we let $$a(A) = \textrm{max}\{n\in {\mathbb Z} \ | \ A_{n} \neq 0\}$$ if $A\neq 0$, and $a(A)=-\infty$ if $A=0$. Now, given an integer $j\geq 0$, let $H_{S_{+}}^{j}(S)$ stand for the $j$-th local cohomology module of $S$ with respect to $S_{+}$. It is well-known that $H_{S_{+}}^{j}(S)$ is a graded module with $H_{S_{+}}^{j}(S)_n=0$ for all $n\gg 0$ (see, e.g., \cite[Proposition 15.1.5(ii)]{B-S}). Thus, $a(H_{S_{+}}^{j}(S))<+\infty$.

\begin{Definition}\rm Maintain the above setting and notations. The \textit{Castelnuovo-Mumford regularity} of the graded ring $S$ is defined as
$$\mathrm{reg}\,S  :=  \mathrm{max}\{a(H_{S_{+}}^{j}(S)) + j \, \mid \, j \geq 0\}.$$
\end{Definition}

A classical instance is when $S=\mathcal{R}(J) = \bigoplus_{i \geq 0} J^{i}$, the Rees algebra of an ideal $J$ in a ring $R$ (recall $\mathcal{R}(J)$ is a finitely generated standard graded $R$-algebra). In particular, we can consider the case where $R$ is local and $J=\M$, the maximal ideal of $R$. Our result is as follows.

\begin{Theorem}\label{cororeg}
	Let $(R, \M)$ be a local ring with infinite residue field and ${\rm depth}\,R>0$, and let $I$ be a reduction of $\M$. Then, 
	$$n_2(I) \leq n_3(I)  = n_1(I) \leq \mathrm{reg}\,\mathcal{R}(\M).$$
	
\end{Theorem}
\begin{proof} First, applying \cite[Theorem 4.8]{Trung} we obtain $\mathrm{reg}\,\mathcal{R}(\M) \geq {\rm r}_{I}(\M)$. Furthermore, according to \cite[Theorem 2.1(ii)]{Rossi-Dinh-Trung} we can write
$$\mathrm{max}\{\mathrm{reg}\,\mathcal{R}(\M), 1\} \geq s(\M).$$ Putting these facts together, we obtain $$\mathrm{reg}\,\mathcal{R}(\M) \geq \mathrm{max}\{s(\M) - 1, {\rm r}_{I}(\M)\}=\alpha.$$ Now the result follows from Theorem \ref{cotred}. \qed
\end{proof}

\medskip

We now pass to our last corollary. It follows from Example \ref{examplesec1} that, if $(R, \M)$ is a regular local ring, then the $n_i(\M)$'s coincide. The result below (which requires the residue field to be infinite) tells us, in particular, what this common value must be.

\begin{Corollary}\label{regular} Let $(R, \M)$ be a regular local ring with infinite residue field, and let $I$ be a reduction of $\M$. Then, $$n_1(I) = n_2(I) = n_3(I) = 0.$$ In particular, $n_1(\M) = n_2(\M) = n_3(\M) = 0$.
\end{Corollary}
\begin{proof} We may assume ${\rm depth}\,R>0$ (otherwise $R$ is just a field). Now, since in the present situation $\M$ is generated by a regular sequence, we have $\mathrm{reg}\,\mathcal{R}(\M) = 0$ (see \cite[Corollary 5.2]{Trung}). Finally, we apply Theorem \ref{cororeg}. \qed
\end{proof}

\begin{Remark}\rm Our Corollary \ref{regular} admits a potentially more general statement, to wit, if $(R, \M)$ is a local ring with infinite residue field such that $\M$ can be generated by a $d$-sequence (in the sense of \cite{Hu}), then likewise we can write $n_1(I) = n_2(I) = n_3(I) = 0$ for any reduction $I$ of $\M$. Indeed, in this case we still have $$\mathrm{reg}\,\mathcal{R}(\M) = 0$$ by \cite[Corollary 5.2]{Trung}. However, we do not believe (although we found neither a proof nor an answer in the literature) that $\M$ can be generated by a $d$-sequence when $R$ is non-regular.
\end{Remark}

\section{Examples}

We conclude this note by  illustrating  our  main results. 

\begin{Example}\label{illust} \rm Fix integers $a, b, c \geq 2$, and consider the 2-dimensional local ring 
$$R =  {\mathbb C}[\![x, y, z, t]\!]/(xy - t^{a+b},\, xz - t^{a+c} + zt^a,\, yz - yt^c+zt^b),$$
which defines a rational triple point, hence is Cohen-Macaulay. Denote $\M  =  (x, y, z, t)R$ and recall from \cite[Example 7.5]{Goto} that $I = (x + y + z, t)R$  is a minimal reduction of $\M$ with 
$$\M^2 = (x^2, xt, y^2, yt, z^2, zt, t^2)R = I\M,$$ that is, $r_I(\M) = 1$. By \cite[(1.3), p.\,595]{Heinzer-Lantz-Shah} we have $s(\M) = 1$. In addition, from the defining relations of $R$ and the inclusion ${\M}^2\subset I$, it is easy to see that $I\colon z=\M=I\colon \M$. Putting these facts together, we can apply Theorem \ref{cotred} to obtain $$n_2(I)=0<  n_3(I)  = n_1(I) = 1.$$
\end{Example}

\medskip

In the following examples, $K$ denotes an infinite field. We next proceed to show, in particular, that Corollary \ref{corindep} only applies to minimal reductions.

\begin{Example} \rm
Let $R = K[\![x, y, z]\!]/(y^3 - xz, x^4 - yz, x^3y^2 - z^2)$, which is a 1-dimensional Cohen-Macaulay local ring. Let $\M = (x, y, z)R$. Some of the aspects considered in this example (e.g., the principal minimal reduction  and the Ratliff-Rush powers of $\M$) are taken from \cite[Example 4.3]{Rossi-Swanson}, where $R$ is taken as $K[\![t^4, t^5, t^{11}]\!]$ for an indeterminate $t$ over $K$. The principal ideal $I = (x)R$ is a minimal reduction of $\M$; more precisely, $\M^{n + 1} = x\M^{n}$ for all $n \geq 3$, whereas $$\M^3 = (x^3, x^2y, xy^2, y^3)R \neq (x^3, x^2y, xy^2)R = x\M^2,$$ and hence ${\rm r}_I(\M) = 3$. By Corollary \ref{corindep} (and the independence of ${\rm r}(\M)$ used in its proof), we obtain $n_2(J)\leq n_1(J) = n_3(J) = {\rm r}(\M)=3$ for every minimal reduction $J$ of $\M$.  In particular, this holds for $I$. A computation also shows $n_2(I)=3$. Therefore,
$$n_2(I)= n_3(I) = n_1(I) = 3.$$

Next, we show that we can have $n_1(I) \neq n_1(L)$ for some non-minimal reduction $L$ of $\M$ (thus confirming that Corollary \ref{corindep} only works for minimal reductions). Note the ideal $L = (x, y)R$ is a reduction of $\M$ such that $\M^2 = (x^2, xy, y^2) = L\M$, i.e., ${\rm r}_L(\M) = 1$. 

Now, as recalled in the proof of Corollary \ref{corindep} we have $s(\M) \leq {\rm r}(\M)$, hence $s(\M)\leq 3$. On the other hand, we can compute  $$\M^2 = (x^2, xy, y^2) \neq (x^2, xy, y^2, z) = \widetilde{\M^2},$$ which gives $s(\M)\geq 3$, and then necessarily $s(\M)=3$. It follows that $s(\M) = 3 > 1 = r_L(\M)$. Finally, applying Theorem \ref{cotred} to the reduction $L$, we obtain $$n_1(L) = s(\M) - 1 =  2 < 3=n_1(I).$$
\end{Example}

\begin{Example} \rm
Let $R = K[\![x, y, z, u, v]\!]/(x^2 + y^5, xy + u^4, xz + v^3)$, which is a $2$-dimensional complete intersection local ring. Consider its maximal ideal $\M = (x, y, z, u, v)R$. According to \cite[Example 3.9]{Strunk}, we can write $$\mathrm{reg}\,\mathcal{R}(\M)=\mathrm{reg}\,\mathcal{G}(\M) = 8,$$ where the first equality is well-known (see \cite[Corollary 3.3]{Trung}). Therefore, for every -- not necessarily minimal -- reduction $I$ of $\M$, Theorem \ref{cororeg} yields $n_2(I) \leq n_3(I)  = n_1(I) \leq 8$ (we have been unable to check whether this bound can be attained). Computations in the particular case $I=\M$ give precisely $$n_2(\M) =n_3(\M) = n_1(\M) =7.$$  \end{Example}

\bigskip

\noindent{\bf Declaration of competing interest.} 
There is no competing interest.

\bigskip

\noindent{\bf Acknowledgements.} The first author was partially supported by the CNPq-Brazil grants 301029/2019-9 and 406377/2021-9. The second author was supported by a CAPES Doctoral Scholarship.


\begin{thebibliography}{99}

\bibitem{B-S} M. Brodmann, R. Sharp, {\it Local Cohomology: An Algebraic Introduction with Geometric Applications}, Cambridge Univ. Press, Cambridge, 1998.



\bibitem{Celikbas} O. Celikbas, T. Kobayashi, {\it On a class of Burch ideals and a conjecture of Huneke
and Wiegand}, Collect. Math. \textbf{73} (2021), 221-236.

\bibitem{Dao} H. Dao, {\it On colon operations and special types of ideals}, Palestine J. Math. \textbf{10} (2021), 383-388.

\bibitem{Dao2} H. Dao, T. Kobayashi, {\it Burch ideals and Burch rings}, Algebra Number Theory \textbf{10} (2020), 2121-2150.


\bibitem{GH} S. Goto, F. Hayasaka, {\it Finite homological dimension and primes associated to integrally closed ideals}, Proc. Amer. Math. Soc. {\bf 130} (2002),
3159-3164.

\bibitem{Goto} S. Goto, K. Ozeki, R. Takahashi, K.-i Watanabe, K.-i Yoshida, {\it Ulrich ideals and modules over two-dimensional
rational singularities}, Nagoya Math. J. {\bf 221} (2016), 69-110.

\bibitem{Harima-Watanabe2} T. Harima, J. Watanabe, {\it The weak Lefschetz property for $\M$-full ideals and componentwise linear ideals}, Illinois J. Math. \textbf{156} (2012), 957-966.

\bibitem{Harima-Watanabe} T. Harima, J. Watanabe, {\it Completely $\M$-full ideals and componentwise linear ideals}, Math. Proc. Camb. Phil. Soc. \textbf{158} (2015), 239-248.




\bibitem{Heinzer-Lantz-Shah} W. Heinzer, D. Lantz, K. Shah, {\it The Ratliff-Rush ideals in a Noetherian ring}, Comm. Algebra \textbf{20} (1992), 591-622.

\bibitem{HRR} W. Heinzer, L. J. Ratliff,  D. E. Rush, {\it Basically full ideals in local rings}, J. Algebra {\bf 250} (2002), 371-396.

\bibitem{Hong} J. Hong, H. Lee, S. Noh, D. E. Rush, {\it Full ideals}, Comm. Algebra, \textbf{37} (2009), 2627-2639.

\bibitem{Huckaba} S. Huckaba, {\it Reduction numbers for ideals of analytic spread one}, J. Algebra \textbf{108} (1987), 503-512.

\bibitem{Hu} C. Huneke, {\it The theory of $d$-sequences and powers of ideals}, Adv. Math. {\bf 46} (1982), 249-279.




\bibitem{Cleto}{C. B. Miranda-Neto}, {\it Free logarithmic derivation modules over factorial domains}, Math. Res. Lett. {\bf 24} (2017), 153-172.



\bibitem{Nag} R. Naghipour, {\it Ratliff–Rush closures of ideals with respect to a Noetherian module}, J. Pure Appl. Algebra \textbf{195} (2005), 167-172.

\bibitem{Ratliff-Rush} L. J. Ratliff, D. E. Rush, {\it Two notes on reductions of ideals}, Indiana Univ. Math. J. \textbf{27} (1978), 929-934.

\bibitem{Rossi-Swanson} M. E. Rossi, I. Swanson, {\it Notes on the behaviour of the Ratliff-Rush filtration}, Contemp. Math. \textbf{331} (2003), 313-328.

\bibitem{Rossi-Dinh-Trung} M. E. Rossi, D. T. Trung, N. V. Trung, {\it Castelnuovo–Mumford regularity and Ratliff–Rush closure}, J. Algebra \textbf{504} (2018), 568-586.


\bibitem{Rush} D. E. Rush, {\it Contracted, $\M$-full and related classes of ideals in local rings}, Glasgow Math. J. \textbf{55} (2013), 669-675.

\bibitem{Strunk} B. Strunk, {\it Castelnuovo–Mumford regularity, postulation numbers, and reduction numbers}, J. Algebra \textbf{311} (2007), 538-550.

\bibitem{Swanson-Huneke} I. Swanson, C. Huneke, {\it Integral Closure of Ideals, Rings and Modules}, London Math. Soc. Lecture Note Ser. {\bf 336}, Cambridge Univ. Press, 2006.

\bibitem{Trung} N. V. Trung, {\it The Castelnuovo regularity of the Rees algebra and the associated graded ring}, Trans. Amer. Math. Soc. \textbf{350} (1998), 2813-2832.

\bibitem{Watanabe1} J. Watanabe, {\it $\M$-full ideals}, Nagoya Math. J.  \textbf{106} (1987), 101-111.

\bibitem{Watanabe2} J. Watanabe, {\it The syzygies of $\M$-full ideals}, Math. Proc. Cambridge Philos. Soc.  \textbf{109} (1991), 7-13.

\bibitem{Watanabe3} J. Watanabe, {\it $\M$-full ideals II}, 
Math. Proc. Cambridge Philos. Soc.  \textbf{111} (1992), 231-240.

\end{thebibliography}
\end{document}